\documentclass{amsart}
\usepackage{setspace}
\usepackage{a4}
\usepackage{amssymb,amsmath,amsthm,latexsym}
\usepackage{amsfonts}
\usepackage{graphicx}
\usepackage{textcomp}
\usepackage{cite}
\usepackage{enumerate}
\usepackage[mathscr]{euscript}
\usepackage{mathtools}
\newtheorem{theorem}{Theorem}[section]

\newtheorem{remark}[theorem]{Remark}
\newtheorem{question}[theorem]{Question}
\title{This is the title}
\usepackage{amssymb}
\usepackage{amsmath}
\usepackage{tikz}
\usepackage{hyperref}
\usepackage{enumerate}
\usepackage{mathtools}
\usepackage{graphicx}
\usepackage{textcomp}
\usetikzlibrary{cd}
\usepackage{fancyhdr}

\begin{document}
	\hrule\hrule\hrule\hrule\hrule
	\vspace{0.3cm}	
	\begin{center}
		{\bf{ALGEBRAIC AND\^{O} DILATION}}\\
		\vspace{0.3cm}
		\hrule\hrule\hrule\hrule\hrule
		\vspace{0.3cm}
		\textbf{K. MAHESH KRISHNA}\\
		Post Doctoral Fellow \\
		Statistics and Mathematics Unit\\
		Indian Statistical Institute, Bangalore Centre\\
		Karnataka 560 059, India\\
		Email: kmaheshak@gmail.com\\
		
		Date: \today
	\end{center}

	\hrule
	\vspace{0.5cm}
	\textbf{Abstract}: We solve the  And\^{o} dilation problem for linear maps on vector space asked by Krishna and Johnson in \textit{[Oper. Matrices, 2022]}. More precisely,  we show that any commuting linear maps on vector space  can be dilated to commuting injective linear maps.

	
	\vspace{0.5cm}
	\hrule 

\section{Introduction}
After a decade of work of Sz.-Nagy \cite{NAGYFOIAS, NAGY}, And\^{o} \cite{ANDO} made a breakthrough result in the dilation theory of contractions on Hilbert space which states as follows.
\begin{theorem}\cite{ANDO}\label{ANDOTHEOREM} (\textbf{And\^{o} Dilation})
	Let $\mathcal{H}$ be a Hilbert space and $T,S:\mathcal{H}\to \mathcal{H}$ be commuting contractions.	Then there exists a Hilbert space $\mathcal{K}$	which contains $\mathcal{H}$ isometrically and a pair of commuting unitaries   $U,V:\mathcal{K}\to \mathcal{K}$ such that 
	\begin{align*}
		T^nS^m=P_\mathcal{H}U^n{V}^m|_\mathcal{H}, \quad \forall n,m \in \mathbb{Z}_+\coloneqq\mathbb{N}\cup\{0\}, 
	\end{align*}
	where $P_\mathcal{H}:\mathcal{K}\to \mathcal{H}$  is the orthogonal projection onto $\mathcal{H}$.
\end{theorem}
After the work of And\^{o}, Parrott  \cite{PARROTT} showed that it is not possible  to improve Theorem \ref{ANDOTHEOREM} for more than two commuting contractions. Later, And\^{o} dilation is derived for commuting contractions on Banach spaces \cite{STROESCU}.  In 2021 (arXiv version), in the paper \cite{KRISHNAJOHNSON}, while continuing the work of  Bhat, De and Rakshit \cite{BHATDERAKSHIT} on dilations of linear maps on vector spaces, Krishna and Johnson \cite{KRISHNAJOHNSON} asked following problem. 
\begin{question}\cite{KRISHNAJOHNSON}\label{KJQUESTION}
\textbf{Whether there is an And\^{o} dilation for linear maps on vector spaces? More precisely, whether commuting linear maps on vector space  can be dilated to commuting bijective linear maps?}
\end{question}
In this paper, we solve Question \ref{KJQUESTION} partially by showing that we can go upto injective linear maps.

\section{Algebraic And\^{o}  Dilation}
 	We first give a different proof Theorem  \ref{BHATTHEOREM}  than given in \cite{BHATDERAKSHIT} which helps us to give a proof of algebraic version of And\^{o} dilation. 
\begin{theorem}\label{BHATTHEOREM}\cite{BHATDERAKSHIT}
(Algebraic Sz.-Nagy Dilation or Bhat-De-Rakshit Dilation)	Let $\mathcal{V}$ be a vector space  and 	$T: \mathcal{V} \to \mathcal{V}$ be a linear map. Then there is a vector space $\mathcal{W}$ containing $\mathcal{V}$ through a natural coordinate injective map and an injective linear map $U: \mathcal{W} \to \mathcal{W}$  such that 
		\begin{align*}
		\text{(Dilation equation)} \quad T^n=P_\mathcal{V}U^n|_\mathcal{V}, \quad \forall n\in \mathbb{Z}_+, 
	\end{align*}
	where $P_\mathcal{V}: \mathcal{W} \to \mathcal{V}$ is a coordinate projection  (idempotent)  onto $\mathcal{V}$.
\end{theorem}
\begin{proof}
Our construction is motivated from the construction of Sz.-Nagy dilation of a contraction on a Hilbert space given  in Chapter 1 of \cite{NAGYFOIAS}.  Given a vector space $\mathcal{V}$, let $I_\mathcal{V}$ be the identity operator on $\mathcal{V}$ and 
$\oplus_{n=0}^{\infty} \mathcal{V}$ be the vector space defined by 
\begin{align*}
	\oplus_{n=0}^{\infty} \mathcal{V}\coloneqq \{ (x_n)_{n=0}^\infty, x_n \in \mathcal{V}, \forall n \in \mathbb{Z}_+, x_n\neq 0 
	\text{ only for finitely many } n's\}.
\end{align*}	
Let $T: \mathcal{V} \to \mathcal{V}$ be a linear map. Define $\mathcal{W}\coloneqq \oplus_{n=0}^{\infty} \mathcal{V}$ and 
\begin{align*}
	&	I:\mathcal{V} \ni x \mapsto (x, 0, \dots ) \in  \mathcal{W},\\
	&	U: \mathcal{W} \ni (x_n)_{n=0}^\infty \mapsto (Tx_0, (I_\mathcal{V}-T)x_0, x_1, x_2 , \dots) \in \mathcal{W},\\
	&	P:\mathcal{W} \ni (x_n)_{n=0}^\infty \mapsto x_0\in \mathcal{V}.
\end{align*}
Then clearly the dilation equation is satisfied. The proof is complete if we show that $U$ is injective. Let  $(x_n)_{n=0}^\infty\in  \mathcal{W}$ be  such that $U(x_n)_{n=0}^\infty=0$. then 
\begin{align*}
	(Tx_0, (I_\mathcal{V}-T)x_0, x_1, x_2 , \dots)=(0,0,0, \dots).
\end{align*}
We then have $x_1=x_2=\cdots=0$ and $Tx_0=(I_\mathcal{V}-T)x_0=0$. Rewriting
\begin{align*}
	x_0=Tx_0=0.
\end{align*}
Thus $(\mathcal{W}, U)$ is an injective linear dilation of $T$.
\end{proof}
Following is the most important result of this paper which we call algebraic  And\^{o} dilation.
\begin{theorem}\label{ALGEBRAICANDOTHEOREM}(\textbf{Algebraic And\^{o} Dilation})
Let $\mathcal{V}$ be a vector space and $T,S: \mathcal{V} \to \mathcal{V}$  be commuting linear maps. Then there is a vector space $\mathcal{W}$ containing $\mathcal{V}$ through a natural coordinate injective map and  injective linear maps $U, V: \mathcal{W} \to \mathcal{W}$ such that 	
\begin{align*}
\text{(Bivariate Dilation equation)} \quad 	T^nS^m=P_\mathcal{V}U^n{V}^m|_\mathcal{V}, \quad \forall n,m \in \mathbb{Z}_+\coloneqq\mathbb{N}\cup\{0\}, 
\end{align*}
	where $P_\mathcal{V}: \mathcal{W} \to \mathcal{V}$ is a coordinate projection  (idempotent)  onto $\mathcal{V}$.
\end{theorem}
\begin{proof}
Our arguments are motivated from original argument for And\^{o} dilation for commuting contractions on Hilbert spaces by  And\^{o} \cite{ANDO}. Define 
$\mathcal{W}\coloneqq \oplus_{n=0}^{\infty} \mathcal{V}$ and 
\begin{align*}
	&	W_1: \mathcal{W} \ni (x_n)_{n=0}^\infty \mapsto (Tx_0, (I_\mathcal{V}-T)x_0, 0, x_1, x_2 , \dots) \in \mathcal{W},\\
	&	W_2: \mathcal{W} \ni (x_n)_{n=0}^\infty \mapsto (Sx_0, (I_\mathcal{V}-S)x_0, 0, x_1, x_2 , \dots) \in \mathcal{W},\\
	&	P:\mathcal{W} \ni (x_n)_{n=0}^\infty \mapsto  x_0\in \mathcal{V}.
\end{align*}
Let $x \in \mathcal{V}$ be such that $(I_\mathcal{V}-T)Sx=0=(I_\mathcal{V}-S)x$. Then 
\begin{align*}
	(I_\mathcal{V}-S)Tx=Tx-STx=Tx-TSx=T(I_\mathcal{V}-S)x=T0=0
\end{align*}
and 
\begin{align*}
(I_\mathcal{V}-T)x	=x-Tx=x-T(Sx)=Sx-TSx=(I_\mathcal{V}-T)Sx=0.
\end{align*}
This obervation says that the map 
\begin{align*}
	v:\{((I_\mathcal{V}-T)Sx, 0,(I_\mathcal{V}-S)x,0): x \in \mathcal{V}\} \to \{((I_\mathcal{V}-S)Tx, 0,(I_\mathcal{V}-T)x,0): x \in \mathcal{V}\}
\end{align*}
defined by 
\begin{align*}
v(I_\mathcal{V}-T)Sx, 0,(I_\mathcal{V}-S)x,0)\coloneqq ((I_\mathcal{V}-S)Tx, 0,(I_\mathcal{V}-T)x,0)
\end{align*}
is a well-defined injective linear map. Clearly $v$ is surjective. We now claim that $v$ can be extended as a bijective linear map (which we again denote by $v$) from $\mathcal{V} \oplus  \mathcal{V}\oplus  \mathcal{V} \oplus  \mathcal{V}  $ to $\mathcal{V} \oplus  \mathcal{V}\oplus  \mathcal{V} \oplus  \mathcal{V}  $. We get two cases. \\
Case (i): $\operatorname{dim}(\mathcal{V})<\infty$.\\
Let $\mathcal{Y}$ be any vector space complement of $\{((I_\mathcal{V}-T)Sx, 0,(I_\mathcal{V}-S)x,0): x \in \mathcal{V}\}$ in $\mathcal{V} \oplus  \mathcal{V}\oplus  \mathcal{V} \oplus  \mathcal{V} $ and $\mathcal{Z}$ be any vector space complement of $\{((I_\mathcal{V}-S)Tx, 0,(I_\mathcal{V}-T)x,0): x \in \mathcal{V}\}$ in $\mathcal{V} \oplus  \mathcal{V}\oplus  \mathcal{V} \oplus  \mathcal{V}$. From the dimension formula for vector spaces, we then get
\begin{align*}
	\operatorname{dim}(\mathcal{V} \oplus  \mathcal{V}\oplus  \mathcal{V} \oplus  \mathcal{V})&=	\operatorname{dim}(\{((I_\mathcal{V}-T)Sx, 0,(I_\mathcal{V}-S)x,0): x \in \mathcal{V}\})+	\operatorname{dim}(\mathcal{Y})\\
	&=	\operatorname{dim}(\{((I_\mathcal{V}-S)Tx, 0,(I_\mathcal{V}-T)x,0): x \in \mathcal{V}\})+\operatorname{dim}(\mathcal{Z})
\end{align*}
Since $\operatorname{dim}(\{((I_\mathcal{V}-T)Sx, 0,(I_\mathcal{V}-S)x,0): x \in \mathcal{V}\})=\operatorname{dim}(\{((I_\mathcal{V}-S)Tx, 0,(I_\mathcal{V}-T)x,0): x \in \mathcal{V}\})$, 
\begin{align*}
	\operatorname{dim}(\mathcal{Y})=\operatorname{dim}(\mathcal{Z}).
\end{align*}
Thus $v$ can be extended bijectively and linearly from  $\mathcal{V} \oplus  \mathcal{V}\oplus  \mathcal{V} \oplus  \mathcal{V}$ to $\mathcal{V} \oplus  \mathcal{V}\oplus  \mathcal{V} \oplus  \mathcal{V}$.\\
Case (i): $\operatorname{dim}(\mathcal{V})=\infty$.\\
Let $\mathcal{Y}$ be any vector space complement of $\{((I_\mathcal{V}-T)Sx, 0,(I_\mathcal{V}-S)x,0): x \in \mathcal{V}\}$ containing the space $\{(0,x,0,0): x\in \mathcal{V}\}$ in $\mathcal{V} \oplus  \mathcal{V}\oplus  \mathcal{V} \oplus  \mathcal{V} $ and $\mathcal{Z}$ be any vector space complement of $\{((I_\mathcal{V}-S)Tx, 0,(I_\mathcal{V}-T)x,0): x \in \mathcal{V}\}$ containing the space $\{(0,x,0,0): x\in \mathcal{V}\}$ in $\mathcal{V} \oplus  \mathcal{V}\oplus  \mathcal{V} \oplus  \mathcal{V} $. Then 
\begin{align*}
\operatorname{dim}(\mathcal{V})	=\operatorname{dim}(\mathcal{V} \oplus  \mathcal{V}\oplus  \mathcal{V} \oplus  \mathcal{V})\geq 	\operatorname{dim}(\mathcal{Y})\geq \operatorname{dim}(\mathcal{V})
\end{align*}
and 
\begin{align*}
	\operatorname{dim}(\mathcal{V})	=\operatorname{dim}(\mathcal{V} \oplus  \mathcal{V}\oplus  \mathcal{V} \oplus  \mathcal{V})\geq 	\operatorname{dim}(\mathcal{Z})\geq \operatorname{dim}(\mathcal{V}).
\end{align*}
Therefore $\operatorname{dim}(\mathcal{Y})=\operatorname{dim}(\mathcal{Z})$ and hence  $v$ can be extended bijectively and linearly from  $\mathcal{V} \oplus  \mathcal{V}\oplus  \mathcal{V} \oplus  \mathcal{V}$ to $\mathcal{V} \oplus  \mathcal{V}\oplus  \mathcal{V} \oplus  \mathcal{V}$.\\
Define $\mathcal{V}^{(4)}\coloneqq \mathcal{V} \oplus  \mathcal{V}\oplus  \mathcal{V} \oplus  \mathcal{V}$. We identify $\mathcal{W}$ and $\mathcal{V} \oplus (\oplus_{n=1}^{\infty} \mathcal{V}^{(4)})$ by the map 
\begin{align*}
	(x_n)_{n=0}^\infty \mapsto (x_0, (x_1, x_2,x_3,x_4), (x_5, x_6, x_7, x_8), \dots)
\end{align*}
Now we define $W:\mathcal{W} \to \mathcal{W}$ by 
\begin{align*}
W(x_n)_{n=0}^\infty \coloneqq (x_0, v(x_1, x_2,x_3,x_4), v(x_5, x_6, x_7, x_8), \dots)
\end{align*}
which becomes bijective linear map with inverse 
\begin{align*}
	W^{-1}(x_n)_{n=0}^\infty \coloneqq (x_0, v^{-1}(x_1, x_2,x_3,x_4), v^{-1}(x_5, x_6, x_7, x_8), \dots)
\end{align*}
We finally define $U\coloneqq WW_1$, $V\coloneqq W_2W^{-1}$  and show that $(\mathcal{W},  (U, V))$ is the required  injective linear dilation of $(T,S)$. Clearly $U$ and $V$ are injective. By induction, we also have the multivariate dilation equation
\begin{align*}
	T^nS^mx=P_\mathcal{V}U^nV^mx, \quad \forall n, m \in \mathbb{Z}_+,  \forall x \in  \mathcal{V}.
\end{align*}
Now we are left only with proving that $ U$ and $V$ commute. Let $(x_n)_{n=0}^\infty \in \mathcal{W}$. Then 
\begin{align*}
	UV(x_n)_{n=0}^\infty&=WW_1W_2W^{-1}(x_n)_{n=0}^\infty\\
&	=WW_1W_2 (x_0, v^{-1}(x_1, x_2,x_3,x_4), v^{-1}(x_5, x_6, x_7, x_8), \dots)\\
&=WW_1(Sx_0, (I_\mathcal{V}-S)x_0, 0,  v^{-1}(x_1, x_2,x_3,x_4), v^{-1}(x_5, x_6, x_7, x_8), \dots) \\
&=W(TSx_0, (I_\mathcal{V}-T)Sx_0, 0, (I_\mathcal{V}-S)x_0, 0,  v^{-1}(x_1, x_2,x_3,x_4), v^{-1}(x_5, x_6, x_7, x_8), \dots) \\
&=(TSx_0, v((I_\mathcal{V}-T)Sx_0, 0, (I_\mathcal{V}-S)x_0, 0), (x_1, x_2,x_3,x_4), (x_5, x_6, x_7, x_8), \dots)\\
&=(STx_0, (I_\mathcal{V}-S)Tx_0, 0, (I_\mathcal{V}-T)x_0, 0), (x_1, x_2,x_3,x_4), (x_5, x_6, x_7, x_8), \dots)\\
\end{align*}
and 
\begin{align*}
	VU(x_n)_{n=0}^\infty&=W_2W^{-1}WW_1(x_n)_{n=0}^\infty=W_2W_1(x_n)_{n=0}^\infty\\	
&=W_2(Tx_0, (I_\mathcal{V}-T)x_0, 0, x_1, x_2 , \dots) \\
&=(STx_0, (I_\mathcal{V}-S)Tx_0, 0, (I_\mathcal{V}-T)x_0, 0), x_1, x_2,x_3,x_4, x_5, x_6, x_7, x_8, \dots).
\end{align*}
Therefore $VU=UV$.
\end{proof}
Theorem \ref{ALGEBRAICANDOTHEOREM} and the works presented in  \cite{EGERVARY, LEVYSHALIT, HALMOS, JOHNSHALIT, SCHAFFER} gives the following problem.
\begin{question}
	\begin{enumerate}[\upshape(i)]
		\item Whether there is an explicit (matrix) construction of algebraic And\^{o} dilation?
		\item Whether there is a Halmos dilation for commuting linear maps on vector spaces?
		\item Whether there is an Egerv\'{a}ry N-dilation for commuting linear maps on vector spaces?
		\item Does Theorem \ref{ALGEBRAICANDOTHEOREM}   holds for more than two commuting linear maps?
		\item Can the dilated  injective linear maps $U,V$ in Theorem \ref{ALGEBRAICANDOTHEOREM}  be improved to bijective linear maps?
	\end{enumerate}
\end{question}
\begin{remark}
	And\^{o} dilation for p-adic magic contractions and self-adjoint morphisms on indefinite inner product modules  over *-rings of characteristic 2 are still open \cite{KRISHNA2, KRISHNA3}.
\end{remark}
\section{Conclusions}
\begin{enumerate}
	\item In 1950, Halmos showed that every contraction on a Hilbert space can be lifted to a unitary \cite{HALMOS}.
	\item In 1953, Sz.-Nagy derived his dilation theorem \cite{NAGY}.
	\item In 1955, Schaffer gave simple proof of Sz.-Nagy dilation result \cite{SCHAFFER}.
	\item In 1963, And\^{o}  showed that Sz.-Nagy dilation holds for two commuting contractions \cite{ANDO}.
	\item In 1973, Stroescu derived And\^{o} dilation for  contractions on Banach spaces \cite{STROESCU}.
	\item In 2021, Bhat, De and Rakshit introduced set theoretic and vector space approach to dilation theory \cite{BHATDERAKSHIT}. Later, Krishna and Johnson continued this study in 2022 \cite{KRISHNAJOHNSON}.
	\item In this paper, we derived And\^{o}  dilation  for linear maps on vector spaces.
\end{enumerate}

 \bibliographystyle{plain}
 \bibliography{reference.bib}

\begin{thebibliography}{10}

\bibitem{ANDO}
T.~And\^{o}.
\newblock On a pair of commutative contractions.
\newblock {\em Acta Sci. Math. (Szeged)}, 24:88--90, 1963.

\bibitem{BHATDERAKSHIT}
B.~V.~Rajarama Bhat, Sandipan De, and Narayan Rakshit.
\newblock A caricature of dilation theory.
\newblock {\em Adv. Oper. Theory}, 6(4):Paper No. 63, 20, 2021.

\bibitem{EGERVARY}
E.~Egerv\'{a}ry.
\newblock On the contractive linear transformations of {$n$}-dimensional vector
  space.
\newblock {\em Acta Sci. Math. (Szeged)}, 15:178--182, 1954.

\bibitem{HALMOS}
Paul~R. Halmos.
\newblock Normal dilations and extensions of operators.
\newblock {\em Summa Brasil. Math.}, 2:125--134, 1950.

\bibitem{KRISHNA3}
K.~Mahesh Krishna.
\newblock Indefinite {H}almos, {E}gervary and {S}z.-{N}agy dilations.
\newblock {\em Preprints:10.20944/preprints202209.0438.v1 29 September}, 2022.

\bibitem{KRISHNA2}
K.~Mahesh Krishna.
\newblock p-adic magic contractions, p-adic von {N}eumann inequality and p-adic
  {S}z.-{N}agy dilation.
\newblock {\em arXiv:2209.12012v1 [math.NT] 24 September}, 2022.

\bibitem{KRISHNAJOHNSON}
K.~Mahesh Krishna and P.~Sam Johnson.
\newblock Dilations of linear maps on vector spaces.
\newblock {\em Oper. Matrices}, 16(2):465--477, 2022.

\bibitem{LEVYSHALIT}
Eliahu Levy and Orr~Moshe Shalit.
\newblock Dilation theory in finite dimensions: the possible, the impossible
  and the unknown.
\newblock {\em Rocky Mountain J. Math.}, 44(1):203--221, 2014.

\bibitem{JOHNSHALIT}
John~E. McCarthy and Orr~Moshe Shalit.
\newblock Unitary {$N$}-dilations for tuples of commuting matrices.
\newblock {\em Proc. Amer. Math. Soc.}, 141(2):563--571, 2013.

\bibitem{PARROTT}
Stephen Parrott.
\newblock Unitary dilations for commuting contractions.
\newblock {\em Pacific J. Math.}, 34:481--490, 1970.

\bibitem{SCHAFFER}
J.~J. Sch\"{a}ffer.
\newblock On unitary dilations of contractions.
\newblock {\em Proc. Amer. Math. Soc.}, 6:322, 1955.

\bibitem{STROESCU}
Elena Stroescu.
\newblock Isometric dilations of contractions on {B}anach spaces.
\newblock {\em Pacific J. Math.}, 47:257--262, 1973.

\bibitem{NAGY}
B\'{e}la Sz.-Nagy.
\newblock Sur les contractions de l'espace de {H}ilbert.
\newblock {\em Acta Sci. Math. (Szeged)}, 15:87--92, 1953.

\bibitem{NAGYFOIAS}
Bela Sz.-Nagy, Ciprian Foias, Hari Bercovici, and Laszlo Kerchy.
\newblock {\em Harmonic analysis of operators on {H}ilbert space}.
\newblock Universitext. Springer, New York, 2010.

\end{thebibliography}

\end{document}